\documentclass{amsart}
\usepackage[margin=3.5cm]{geometry}
\usepackage{amssymb,amsthm}
\newtheorem{theorem}{Theorem}[section]
\newtheorem{intro}{Theorem}
\newenvironment{introtheorem}[1]
  {\renewcommand\theintro{#1}\intro}
  {\endintro}
\newtheorem{corollary}[theorem]{Corollary}
\newtheorem*{corollary*}{Corollary}
\newtheorem{lemma}[theorem]{Lemma}
\theoremstyle{remark}
\newtheorem*{remark}{Remark}
\newcommand{\tO}{\mathtt 0}       
\newcommand{\tL}{\mathtt 1}       
\allowdisplaybreaks
\begin{document}
\title[Discrepancy of the Van der Corput sequence]%
{Discrepancy results for the Van der Corput sequence}
\author{Lukas Spiegelhofer}
\address{Institut f\"ur diskrete Mathematik und Geometrie,
Technische Universit\"at Wien,
Wiedner Hauptstrasse 8--10, 1040 Wien, Austria}
\thanks{The author acknowledges support by 
projects F5502-N26 and F5505-N26 (FWF),
which are part of the Special Research Program
``Quasi Monte Carlo Methods: Theory and Applications''.
Moreover, the author thanks the Erwin Schr\"odinger Institute for Mathematics and Physics where part of this paper was written during his visit for the
programme ``Tractability of High Dimensional Problems and Discrepancy''.
}
\keywords{Van der Corput sequence, irregularities of distribution, digit reversal}
\subjclass[2010]{Primary: 11K38, 11A63; Secondary: 11K31}


\begin{abstract}
Let $d_N=ND_N(\omega)$ be the discrepancy of the Van der Corput sequence in base $2$.
We improve on the known bounds for the number of indices $N$ such that $d_N\leq \log N/100$.
Moreover, we show that the summatory function of $d_N$ satisfies an exact formula involving a $1$-periodic, continuous function.
Finally, we show that $d_N$ is invariant under digit reversal in base $2$.
\end{abstract}
\maketitle
\section{Introduction}
Every nonnegative integer $n$ admits a unique expansion $n=\sum_{i=0}^\nu \varepsilon_i 2^i$ such that $\nu=0$ or $\varepsilon_\nu\neq 0$.
We let $\varepsilon_i(n)$ denote the $i$-th digit in base $2$.
The Van der Corput sequence is defined via the \emph{radical inverse} of $n$ in base $2$: define $\omega_n=\sum_{i=0^\nu}\varepsilon_i(n)2^{-i-1}$.

Let $x=(x_n)_{n\geq 0}$ be a sequence in 
$[0,1)$.                                 
The \emph{discrepancy} $D_N(x)$ of $x$ is defined by
\[D_N(x)=\sup_{0\leq a\leq b\leq 1}\bigl \lvert
A_N(x,a,b)/N-b \bigr \rvert\]
for $N\geq 1$, where $A_N(x,a,b)=\left\lvert\{n<N:a\leq x_i<b\}\right\rvert$.
Moreover, we set $D_0(x)=0$.
The \emph{star-discrepancy} (or discrepancy at the origin) of a sequence $x$ in   
$[0,1)$                                                                           
is defined by
$D^*_N(x)=\sup_{0\leq b\leq 1}\lvert A_N(x,0,b)/N-N \rvert,$
for $N\geq 1$, and we set $D^*_0(x)=0$.

In this paper, we are concerned with the discrepancy of the Van der Corput sequence. We define 
\[d_N=ND_N(\omega),\]
and we will use this notation throughout this paper.
It is well known~\cite[Th\'eor\`eme 1]{BF1978} that the star discrepancy of the Van der Corput sequence equals its discrepancy: we have $D^*_N(\omega)=D_N(\omega)$.
The Van der Corput sequence is a \emph{low discrepancy sequence}, that is, we have $d_N\ll \log N$. More precise results are known:
B\'ejian and Faure~\cite{BF1978} proved the following theorem.
\begin{introtheorem}{A}
\[d_N\leq \frac 13\log_2 N+1\]
for all $N\geq 1$; moreover
\[\limsup_{N\rightarrow\infty} \left(d_N-\frac 13\log_2 N\right)=\frac 49+\frac 13\log_2 3,\]
where $\log_2$ denotes the logarithm in base $2$.
\end{introtheorem}
In the proof of these statements, they implicitely show that $d_N$ is bounded above by the polygonal path connecting the first maxima on the intervals $I_k=[2^{k-1},2^k]$, given by the points $\bigl(\frac 13\left(2^{k+1}+(-1)^k\right),\frac k3+\frac 79+(-1)^k/(9\cdot 2^{k-1})\bigr)$.
This should be compared to the argument given by Coons and Tyler~\cite{CT2014} concerning Stern's diatomic sequence (also called Stern--Brocot sequence), see also the paper by Coons and Spiegelhofer~\cite{CS2017} and the recent paper by Coons~\cite{C2017}.

Concerning the ``usual'' order of magnitude of the discrepancy of the Van der Corput sequence, Drmota, Larcher and Pillichshammer~\cite[Theorem~2]{DLP2005} proved a central limit theorem for $d_N$.
\begin{introtheorem}{B}
For every real $y$, we have
\begin{equation}\label{eqn_DLP}
\frac 1M\left\lvert\left\{
N<M:d_N\leq \frac 14\log_2 N+y\frac 1{4\sqrt{3}}\sqrt{\log_2 N}
\right\}\right\rvert
=\Phi(y)+o(1)
,\end{equation}
where $\Phi(y)=\frac 1{\sqrt{2\pi}}\int_{-\infty}^y e^{-t^2/2}\,\mathrm dt$.
\end{introtheorem}

We note that this implies in particular that $d_N$ is usually of order $\log N$.
More precisely, letting $A_{M,y}$ denote the expression on the left hand side of~\eqref{eqn_DLP},
we trivially have $A_{M,y'}\leq A_{M,y}$ if $y'\leq y$.
This implies, for any sequence $(y_M)_{M\geq 1}$ of reals such that $y_M\rightarrow -\infty$ for $m\rightarrow\infty$, that
\[\lim_{M\rightarrow\infty}A_{M,y(M)}\leq \lim_{M\rightarrow\infty}A_{M,y}=\Phi(y)\]
for all real $y$, therefore this limit equals $0$.
In particular, if $\delta<1/4$, the number of integers $N<M$ such that $d_N\leq \delta\log N$ is $o(M)$.

Bounds of this type, with an explicit error term, had been proved earlier: S\'os~\cite{S1983} proved such a statement for $\{n\alpha\}$-sequences, more generally Tijdeman and Wagner~\cite{TW1980} showed that any sequence in
$[0,1)$
has almost nowhere small discrepancy. More specifically, they proved the following theorem.

\begin{introtheorem}{C}
Let $\xi$ be a sequence in    
$[0,1)$.                      
Let $M$ and $N$ be integers with $M\geq 0$ and $N>1$.
Then $D_n(\xi)<\log N/100$ for at most $2 N^{5/6}$ integers $n$ with
$M<n\leq M+N$.
\end{introtheorem}

In fact, it follows from Lemma~2 in their paper~\cite{TW1980} that the exponent $5/6$ can be replaced by an arbitrarily small positive value if we demand an arbitrarily small constant in place of $1/100$.

\begin{corollary*}
Let $\xi$ be a sequence in    
$[0,1)$.                      
For each $\varepsilon>0$ there exists a constant $\delta>0$ such that for all integers $M\geq 0$ and $N>1$
we have $D_n(\xi)<\delta\log N$ for at most $2N^{\varepsilon}$ integers $n$ with
$M<n\leq M+N$.
\end{corollary*}

We proceed to the statement of our results.
\section{Results}
We wish to show that the constant $5/6$ in Theorem~C can be improved at least for the Van der Corput sequence.

\begin{theorem}\label{thm_strong_irregularity}
For all large $N$, the number of $n<N$ satisfying $d_N\leq\log N/100$ is bounded above by $N^{0.183}$.
\end{theorem}

Moreover, Tijdeman and Wagner~\cite[Theorem~3]{TW1980}
showed that for infinitely many $N$ we have $d_n\leq \log N/100$ for more than $N^{1/21}$ integers $n\in[1,N]$.
We wish to improve on the exponent $1/21$.
\begin{theorem}\label{thm_strong_irregularity2}
For all large $N$, the number of $n<N$ satisfying $d_N\leq \log N/100$ is bounded below by $N^{0.056}$.
\end{theorem}

It would be interesting to determine, for each given $\varepsilon>0$, the exact ``exponent of strong irregularity'' of the Van der Corput sequence.
That is, determine the infimum of $\eta$ such that the number of $n<N$ satisfying $d_N\leq \varepsilon \log N$ is bounded by $N^\eta$, for all large $N$.
By the above results this infimum, for $\varepsilon=1/100$, lies in $[0.056,0.183]$.
We leave this as an open question.
Next, we consider partial sums
\[S(N)=d_1+\cdots+d_N.\]

It was shown by B\'ejian and Faure~\cite{BF1978} that
\[\frac 1N\sum_{k=1}^Nd_k=\frac{\log_2N}{4}+O(1),\]
where $\log_2 N$ denotes the base-$2$ logarithm of $N$.
We are interested in the error term appearing in this expression.
It turns out that there exist an exact formula involving a $1$-periodic, continuous function (see, for example, the papers by Delange~\cite{D1975} and Flajolet et al.~\cite{FGKPT1994}).

\begin{theorem}\label{thm_one_periodic}
There exists a continuous, $1$-periodic function $\psi:\mathbb R\rightarrow \mathbb R$ such that
\begin{equation}\label{eqn_one_periodic}
  \frac 1NS(N)=\frac{\log_2N}{4}+\frac {d_N}{2N}+\psi(\log_2 N).
\end{equation}
The function $\psi$ is uniquely determined.
\end{theorem}
In particular, we obtain the boundedness result of the error term given by B\'ejian and Faure.
Our third result is concerned with \emph{digit reversal}:
If $\varepsilon_\nu\cdots \varepsilon_0$ is the proper binary expansion of $n$,
we define $n^R=\sum_{0\leq i\leq \nu}\varepsilon_{\nu-i}2^i$.
Then the following theorem holds.
\begin{theorem}\label{thm_reflection}
Assume that $\alpha,\beta,\gamma$ are complex numbers and that the sequence $x$ satisfies $x_{2n}=x_n$ and $x_{2n+1}=\alpha x_n+\beta x_{n+1}+\gamma$ for $n\geq 1$.
Then for $n\geq 1$ we have
\[x_n=x_{n^R}.\]
\end{theorem}
This theorem generalizes Theorem~2.1 in the paper~\cite{S2017} by the author, see also Morgenbesser and the author~\cite{MS2012} and the recent paper by the author~\cite{S2017b}.
We obtain the following, somewhat curious, corollary.
\begin{corollary}
\[ND_N(\omega)=N^RD_{N^R}(\omega).\]
\end{corollary}
We note, however, that this digit reversal property seems to be restricted to base $2$. That is, the Van der Corput sequence in base $q$, where $q\geq 3$, does not seem to satisfy an analogous property with respect to digit reversal in base $q$. We refer the reader to~\cite{F1981,F2005,LP2003,P2004} concerning results on the discrepancy and diaphony of digital sequences. Among these one can find explicit formulas for the star discrepancy analogous to~\eqref{eqn_d_explicit}.
For illustration, we list the first values of $d_N=ND_N(\omega)$.
\[\begin{array}{ccccccccccccccccc}
N&0&1&2&3&4&5&6&7&8&9&10&11&12&13&14&15\\
d_N&0&1&1&\frac 32&1&\frac 74&\frac 32&\frac 74&1&\frac{15}{8}&\frac 74&\frac{17}8&\frac 32&\frac{17}8&\frac 74&\frac {15}8\\[5mm]
N&16&17&18&19&20&21&22&23&24&25&26&27&28&29&30&31\\
d_N&1&\frac{31}{16}&\frac {15}8&\frac{37}{16}&\frac 74&\frac{39}{16}&
\frac{17}8&\frac{37}{16}&\frac 32&\frac{37}{16}&\frac {17}8&\frac{39}{16}&\frac 74&\frac{37}{16}&\frac{15}8&\frac{31}{16}
\end{array}
\]
Apart from the identity $d_N=d_{2^k-N}$, which is valid for $2^{k-1}\leq N\leq 2^k$ and which can be shown easily by induction, we see the notable identity
$d_{19}=d_{25}$. Note that $19^R=25$.

The remainder of this paper is dedicated to the proofs of our results.
\section{proofs}
We will use the following explicit formula due to B\'ejian and Faure~\cite{BF1978}.
\begin{equation}\label{eqn_d_explicit}
d_N=\sum_{j=1}^\infty \bigl \lVert N/2^j\bigr \rVert.
\end{equation}
Based on this result B\'ejian and Faure proved that $d_N$ satisfies the following recurrence:
\begin{equation}\label{eqn_d_recurrence}
d_0=0, \quad d_1=1,\quad d_{2N}=d_N, \quad d_{2N+1}=\frac{d_N+d_{N+1}+1}2,
\end{equation}
which is valid for all $N\geq 0$.

We note that $(d_n)_{n\geq 0}$ is a $2$-regular sequence in the sense of Allouche and Shallit~\cite{AS1992}.
Moreover, the recurrence is of the discrete divide-and-conquer type~\cite{DS2013,GH2005}.
\subsection{Proof of Theorems~\ref{thm_strong_irregularity} and~\ref{thm_strong_irregularity2}}
In order to prove these theorems, we state a couple of lemmas.
We let $\lvert N\rvert_{\tO\tL}$ denote the number of occurrences of $\tO\tL$ in the binary expansion of $N$.
\begin{lemma}
We have
\begin{equation}\label{eqn_DiscrBlock}
\frac 12 \lvert N\rvert_{\tO\tL}\leq d_N\leq 2 \lvert N\rvert_{\tO\tL}.
\end{equation}
\end{lemma}
\begin{proof}
We use the formula $d_N=\sum_{j=1}^\infty \left\lVert \frac N{2^j}\right\rVert$.
Assume that $m=\lvert N\rvert_{\tO\tL}$.
For $0\leq k<m$ let $a_k$ be the index corresponding to the beginning of the $k$-th block of $\tL$s, and $b_k$ be the index corresponding to the end.
We prove the first inequality first.
We have
\[\sum_{j\geq a_0+2}\left\lVert N/2^{j}\right\rVert
=\sum_{j\geq 0}\left\lvert N/2^{a_0+2+j}\right\rvert
\geq\sum_{j\geq 0}\left\lvert 1/2^{2+j}\right\rvert=1/2,\]
moreover for $0\leq k<m-1$
\[\left\lVert N/2^{b_k+1}\right\rVert
\geq \left\lVert 1/2+1/8+1/16+\cdots\right\rVert
= 1/4\]
and for $1\leq k<m$
\[\left\lVert N/2^{a_k+2}\right\rVert
\geq \left\lVert 1/4\right\rVert=1/4.\]
To conclude the proof of the first inequality, we note that the indices $b_k+1$ and $a_k+2$ are pairwise different.

As for the second inequality,
we bound the contribution of each block of $\tL$s by $2$ as follows.
For simplicity of the argument, we set $b_{-1}=\infty$.
We have
\begin{align*}
d_N&=\sum_{j=1}^\infty \left\lVert \frac N{2^j}\right\rVert
=
\sum_{-1\leq k<m-1}
\left(
\sum_{j=a_{k+1}+2}^{b_k}
\left\lVert \frac N{2^j}\right\rVert
+
\sum_{j=b_{k+1}+1}^{a_{k+1}+1}
\left\lVert \frac N{2^j}\right\rVert
\right).
\end{align*}
The summands are bounded above by geometric series with quotient $q=1/2$, which yields the second inequality.
\end{proof}
We note that the constant $2$ is optimal, which can be seen by considering integers having the binary expansion $(\tO^s\tL^s)^k$ and letting $s\rightarrow\infty$.
The constant $1/2$ probably can be improved, but not beyond $2/3$, which follows by considering integers of the form $(\tO\tL)^k$ and letting $k\rightarrow\infty$.
The next lemma is concerned with counting occurrences of $\tO\tL$ in the binary expansion.
\begin{lemma}\label{lem_blockcount}
For $k,\ell\geq 0$ set
\[a_{k,\ell}=\left\lvert\left\{n\in[2^k,2^{k+1}):\lvert n\rvert_{\tO\tL}=\ell\right\}\right\rvert.\]
Then
\[a_{k,\ell}={k+1 \choose 2\ell-1}\]
\end{lemma}
\begin{proof}
We are interested in the set $\mathcal A$ of integers                 
$n\in[2^k,2^{k+1})$                                                   
having exactly $\ell$ blocks of consecutive $\tL$s.
We define a bijection $\varphi$ from $\mathcal A$ onto the set of $2\ell-1$-element subsets of $\{0,\ldots,k\}$ as follows.
The binary expansion $\varepsilon_k\cdots \varepsilon_0$ of $n$
consists of $\ell$ blocks of consecutive $\tL$s and $\ell-1$ or $\ell$ blocks of consecutive $\tO$s.
Let $\varphi(n)$ consist of those indices $i\in\{0,\ldots,k\}$ corresponding to the rightmost element of a block of $\tL$s or to the rightmost element of one of the first $\ell-1$ blocks of $\tO$s.
It is clear how to construct the inverse function.
\end{proof}

We are interested in the quantity  
\[A_{k,\varepsilon}=\left\lvert\left\{N\in[2^k,2^{k+1}):d_N\leq \varepsilon \log N\right\}\right\rvert.\]   
By~\eqref{eqn_DiscrBlock} and Lemma~\ref{lem_blockcount}
we have
\begin{equation}\label{eqn_A_estimate}
\begin{aligned}
A_{k,\varepsilon}&\leq
\left\lvert\left\{N\in[2^k,2^{k+1}):\lvert N\rvert_{\tO\tL}\leq 2\varepsilon \log 2^{k+1}\right\}\right\rvert
\\&=\sum_{\ell=0}^{2\varepsilon(k+1)\log 2}a_{k,\ell}
=\sum_{\ell=0}^{2\varepsilon(k+1)\log 2}{k+1\choose 2\ell-1}
\end{aligned}
\end{equation}
and
\begin{equation}\label{eqn_A_estimate2}
\begin{aligned}
A_{k,\varepsilon}&\geq 
\left\lvert\left\{N\in[2^k,2^{k+1}):\lvert N\rvert_{\tO\tL}\leq (\varepsilon/2) \log 2^k\right\}\right\rvert
\\&=\sum_{\ell=0}^{(\varepsilon/2) k\log 2}a_{k,\ell}
=\sum_{\ell=0}^{(\varepsilon/2)k\log 2}{k+1\choose 2\ell-1}.
\end{aligned}
\end{equation}
We are therefore interested in large deviations of the binomial distribution.
To this end, we state the following lemma.
\begin{lemma}\label{lem_BinomEstimate}
Assume that $k,\ell \geq 1$ are integers,  
$\alpha,\beta\in(0,1/e]$ real numbers, where $e=2.71828\ldots$, and $\alpha k\leq \ell\leq \beta k$.  
Then
\[\frac 1{3\sqrt{k}}\left(\frac{\alpha^{-\alpha}}{(1-\alpha)^{1-\beta}}\right)^k\leq
{k\choose \ell}\leq \left(\frac {\beta^{-\beta}}{(1-\beta)^{1-\alpha}}\right)^k.\]
\end{lemma}
\begin{proof}
For all $n\geq 1$, we have the estimate (see Robbins~\cite{R1955})
\[\sqrt{2\pi}n^{n+1/2}e^{-n}\leq n!\leq e n^{n+1/2}e^{-n}.\]
Noting that the maximum of $\ell\mapsto \ell^{-1/2}(k-\ell)^{-1/2}$ for $1\leq \ell\leq k\alpha$ is attained at $\ell=1$, it follows that
\[{k\choose \ell}\leq \frac{\sqrt{k}}{\sqrt{\ell}\sqrt{k-\ell}}
\frac{e}{2\pi}
\frac{k^k}{\ell^\ell(k-\ell)^{k-\ell}}
\leq
\frac{k^k}{\ell^\ell(k-\ell)^{k-\ell}}
=\left(\frac{k}{\ell}\right)^\ell\left(1+\frac{\ell}{k-\ell}\right)^{k-\ell}
.\]
Since $\ell\mapsto (k/\ell)^\ell$ is increasing for $\ell\in[1,n/e]$, it follows that
\[{k\choose \ell}\leq \left(\frac{k}{k\beta}\right)^{k\beta}\left(1+\frac{k\beta}{k(1-\beta)}\right)^{k(1-\alpha)}.\]
This implies the second inequality.
The proof of the first inequality is similar.
\end{proof}
In order to prove Theorems~\ref{thm_strong_irregularity} and~\ref{thm_strong_irregularity2}, we combine Lemma~\ref{lem_BinomEstimate} and the estimates~\eqref{eqn_A_estimate},~\eqref{eqn_A_estimate2}.
From~\eqref{eqn_A_estimate} and the lemma, we obtain for large $k$, $\beta=4\log 2/100$ and $\alpha=\beta-\delta$, where $\delta>0$ is small,
\begin{align*}
A_{k,1/100}
&\leq
\bigl(2/100(k+1)\log 2+1\bigr)
\max_{\ell\leq \frac 2{100}(k+1)\log 2}{k+1\choose 2\ell-1}
\\&\leq
\bigl(2/100(k+1)\log 2+1\bigr)
\bigl(\beta^{-\beta}/(1-\beta)^{1-\alpha}\bigr)^{k+1}
\end{align*}
for all $k\geq 1$.
Since \[\bigl \lvert\bigl\{n<N:d_n\leq \log n/100\bigr\}\bigr \rvert
\leq A_{0,1/100}+\cdots+A_{L,1/100},\]
where $2^L\leq N<2^{L+1}$, we easily obtain the first theorem by noting that
$\log\bigl(\beta^{-\beta}/(1-\beta)^{1-\beta}\bigr)/\log(2)<0.183$.

To prove Theorem~\ref{thm_strong_irregularity2}, 
we note that for large $k$ we obtain from~\eqref{eqn_A_estimate2} and Lemma~\ref{lem_BinomEstimate}, setting $\beta=\log 2/100$ and $\alpha=\beta-\delta$,
\[A_{k,1/100}\geq {k+1\choose 2\lfloor (1/200)k\log 2\rfloor-1}\geq \frac 1{3\sqrt{k+1}}\left(\frac{\alpha^{-\alpha}}{(1-\alpha)^{1-\beta}}\right)^k.
\]
This implies the statement of Theorem~\ref{thm_strong_irregularity2},
noting that
$\log\bigl(\beta^{-\beta}/(1-\beta)^{1-\beta}\bigr)/\log(2)>0.056$.
\subsection{Proof of Theorem~\ref{thm_one_periodic}}
We define $S'(N)=S(N)-d_N/2 = d_1+\cdots+d_{N-1}+d_N/2$.
By splitting the sum into even and odd indices and using the recurrence~\eqref{eqn_d_recurrence}, we obtain
\begin{equation}\label{eqn_S_invariance}
\begin{aligned}
S'(2N)&=\sum_{k=1}^{N-1}d_{2k}+\sum_{k=0}^{N-1}d_{2k+1}+\frac{d_{2N}}2
=S'(N)+\frac 12\sum_{k=0}^{N-1}(d_k+d_{k+1}+1)
\\&=S'(N)+\frac 12\sum_{k=1}^{N-1}d_k+\frac 12\sum_{k=1}^{N}d_k+\frac N2
=2S'(N)+\frac N2.
\end{aligned}
\end{equation}

Define
\[R(N)=\frac 1NS'(N)-\frac 14\log_2 N.\]
By a simple calculation using~\eqref{eqn_S_invariance} we obtain
\begin{equation}\label{eqn_R_invariance}
R(2N)=R(N).
\end{equation}
We may therefore define a $1$-periodic function $\psi$
defined on the set $\{\log_2 N:N\in\mathbb N\}+\mathbb Z$ as follows:
if $x=\log_2 N+\ell$, where $N$ is odd and $\ell\in\mathbb Z$, we set $\psi(x)=R(N)$. 
Using the identity $R(2N)=R(N)$, it is easy to see that this is well-defined,
moreover~\eqref{eqn_one_periodic} holds.

We need to show that $\psi$ has a continuous continuation to $\mathbb R$.
Since the points $\{\log_2 N\}$ are dense in 
$[0,1)$                                      
such a continuation is necessarily unique.

We define auxiliary functions $F_k:[2^{k-1},2^k]\rightarrow\mathbb R$ by $F_k(x)=S'(\lfloor x\rfloor)$.
Note that by Theorem~A the maximal height of a jump of $F_k$ is $k/3+O(1)$.
We define                     
$\psi_k:[0,1)\rightarrow\mathbb R$  
in such a way that
\[\psi_k(\{\log_2(x)\})=\frac 1xF_k(x)-\frac 14\log_2 x
\quad\textrm{for}\quad 2^{k-1}\leq x<2^k.\]
Moreover, set $\psi_k(1)=F_k(2^k)/2^k-k/4$.
We have $\psi_k(0)=\psi_k(1)=1/2$ by~\eqref{eqn_R_invariance}.
Note that each  
$z\in [0,1)$    
is hit exactly once by the function $\{\log_2(x)\}$, therefore $\psi_k$ is uniquely determined.
Moreover the height of the jumps of $\psi_k:[0,1]\rightarrow \mathbb R$ is bounded by $O(k/2^k)$.
We first show pointwise convergence of the sequence $(\psi_k)_k$.
The statement is clear for $z=\{\log_2 N\}$ and also for $z=1$. 
Assume that $z\in[0,1)$ is not of this form.
Choose, for each $k\geq 1$,                                                
\[N_k=\max\bigl\{N\in[2^{k-1},2^k):\left\{\log_2 N_k\right\}\leq z\bigr\}.\]   
We consider the sequence of values $\psi_k(\{\log_2 N_k\})$.
Note that $N_{k+1}\in\{2N_k,2N_k+1\}$.
Trivially, we have $\lvert \psi_{k+1}(\{\log_2(2N_k)\})-\psi_k(\{\log_2 N_k\})\rvert=0$.
By~\eqref{eqn_S_invariance} we have
\begin{align*}
&\lvert \psi_{k+1}(\{\log_2(2N_k+1)\})-\psi_k(\{\log_2 N_k\})\rvert
\\&=\frac 1{2N_k+1}S'(2N_k+1)-\frac 14\log_2(2N_k+1)
-\left(\frac 1{N_k}S'(N_k)-\frac 14\log_2 N_k\right)
\\&=\frac 1{2N_k}\left(2S'(N_k)+\frac {N_k}2+\frac {d_{2N_k}+d_{2N_k+1}}2\right)
+\left(\frac 1{2N_k+1}-\frac 1{2N_k}\right)S'(2N_k+1)
\\&+\frac 14\bigl(\log_2 (2N_k)-\log_2(2N_k+1)\bigr)
-\frac 14\bigl(\log_2(2N_k)-\log_2 N_k\bigr)-\frac{S'(N_k)}{N_k}
\\&=\frac {d_{2N_k}+d_{2N_k+1}}{4N_k}-\frac{S'(2N_k+1)}{2N_k(2N_k+1)}+
\frac 14\bigl(\log_2(2N_k)-\log_2(2N_k+1)\bigr)
.
\end{align*}
Using the estimate $S'(N_k)=O(N_k\log(N_k))$, which follows from Theorem~A, we obtain
\begin{equation}\label{eqn_consecutive}
\bigl \lvert \psi_{k+1}(\{\log_2 N_{k+1}\})-\psi_k(\{\log_2 N_k\})\bigr \rvert
\leq
C'\frac{\log N_k}{N_k}
\leq C\frac{k}{2^k},
\end{equation}
where the constant $C$ is independent of $z$.

Moreover, let   
$x\in [2^{k-1},2^k)$ 
be such that $z=\{\log_2 x\}$.
Note that $N_k<x<N_k+1$.
We have
\begin{equation}\label{eqn_between}
\begin{aligned}
\bigl \lvert \psi_k(\{\log_2 x\})-\psi_k(\{\log_2 N_k\})\bigr \rvert
&\leq \left\lvert \frac 1xS'(\lfloor x\rfloor)-\frac 1{N_k}S'(N_k) \right\rvert
+\frac 14 \bigl \lvert \log_2 x - \log_2 N_k\bigr \rvert
\\&\leq S'(N_k)\left(\frac 1{N_k}-\frac 1x\right)+\frac 1{4N_k}
\leq C''\frac {\log N_k}{N_k}\leq C\frac{k}{2^k},
\end{aligned}
\end{equation}
where the constant $C$, without loss of generality, is the same as in~\eqref{eqn_consecutive}.

We define $K_\ell=C\sum_{i\geq \ell}\frac{i}{2^i}$.
We note that $K_\ell\rightarrow 0$ as $\ell\rightarrow\infty$.
Let $I_k$ be the symmetric interval of length $2K_k$ around $\psi_k (\{\log_2 N_k\})$.
By~\eqref{eqn_consecutive} and the triangle inequality we have $I_{k+1}\subseteq I_k$, moreover~\eqref{eqn_between} implies $\psi_k(z)\in I_k$.
By the nested intervals theorem the sequence $(\psi_k)_{k\geq 1}$ converges pointwise to a function that we call $\psi$.
Since both of $\psi_k(z)$ and $\psi(z)$ lie in the interval $\psi_k(\{\log_2 N_k\})\pm K_k=I_k$, the number $\psi_k(z)$ lies in the interval $\psi(z)\pm 2K_k$ for all 
$z\in[0,1)$ 
and $k\geq 1$, therefore the sequence $(\psi_k)_{k\geq 1}$ of functions converges uniformly to $\psi$.
We need to show continuity of $\psi$.
Let $z\in[0,1]$ and assume that $\varepsilon>0$.
Choose $k$ so large that the height of the jumps of $\psi_k$ is bounded by $\varepsilon/3$ and also such that $\sup_{0\leq y\leq 1}\lvert \psi(y)-\psi_k(y)\rvert<\varepsilon/3$.
Let $\delta$ be so small that $\psi_k$ has at most one jump in the interval $[z-\delta,z+\delta]\cap[0,1]$.
Application of the triangle inequality and noting that $\psi_k$ is nonincreasing between the jumps finishes the proof of continuity.
Moreover, we have $\psi_k(0)=\psi_k(1)=1/2$, therefore the continuation to $\mathbb R$ is continuous. 
\qed

\begin{remark}
We note that similar reasoning can be applied to Stern's diatomic sequence defined by $s_1=1$, $s_{2n}=s_n$ and $s_{2n+1}=s_n+s_{n+1}$ for $n\geq 1$.
The partial sums $S'(N)=s_1+\cdots+s_{N-1}+s_N/2$ satisfy $S'(2N)=3S'(N)$,
moreover the maximum of $s_n$ on dyadic intervals  
$[2^k,2^{k+1})$                                    
is $F_{k+2}$, where $F_k$ is the $k$-th Fibonacci number (see Lehmer~\cite{L1929} and Lind~\cite{L1969}).
We obtain a representation of the partial sums $S(N)=s_1+\cdots+s_N$:
\[
S_N=N^{\log_23}\psi(\log_2 N)+\frac{s_N}2,
\]
where $\psi$ is continuous and $1$-periodic.
\end{remark}
\subsection{Proof of Theorem~\ref{thm_reflection}}
The proof is an adaption of the proof of~\cite[Theorem~2.1]{S2017}.
The central property that we will need in our proof is given by the following lemma.
\begin{lemma}\label{lem_reflection_matrices}
Let
\[
  A=\left(\begin{matrix}1&0&0\\\alpha&\beta&\gamma\\0&0&1\end{matrix}\right),
\qquad
  B=\left(\begin{matrix}\alpha&\beta&\gamma\\0&1&0\\0&0&1\end{matrix}\right),
\]
  \[v=\left(\begin{matrix}\alpha&\beta&\gamma\end{matrix}\right)
    \quad\textrm{and}\quad
  w=\left(\begin{matrix}1&1&1\end{matrix}\right)^T.
    \]
Then the following identities for $1\times 3$-matrices hold.
\begin{equation*}
\begin{array}{l@{}l@{\,}r@{\,}r@{\,}l@{\,}l@{\,}l@{\,}r@{\,}l@{\,}l@{\,}r@{\,}r@{\,}l}
  v&AAA&=&  0\cdot&v&AA  &+&  (\beta^2+\beta+1)&v&A  &+&  (-\beta^2-\beta)&v,\\[1mm]
  w^T&A^TA^TA^T&=&  0\cdot&w^T&A^TA^T  &+&  (\beta^2+\beta+1)&w^T&A^T&+& (-\beta^2-\beta)&w^T,\\[4mm]
  v&AAB&=&  (\beta+1)&v&AB  &+&  (-\beta)&v&B  &+&  0\cdot&v,  \\[1mm]
  w^T&A^TA^TB^T&=&  (\beta+1)&w^T&A^TB^T  &+&  (-\beta)&w^T&B^T  &+& 0\cdot&w^T,\\[4mm]
  v&ABA&=&  (\beta+1)&v&BA  &+&  0\cdot&v&A  &+&  (-\beta)&v,\\[1mm]
  w^T&A^TB^TA^T&=&  (\beta+1)&w^T&B^TA^T  &+&  0\cdot&w^T&A^T  &+& (-\beta)&w^T,\\[4mm]
  v&ABB&=&  (\alpha+1)&v&AB  &+&  (-\alpha)&v&A  &+&  0\cdot&v,\\[1mm]
  w^T&A^TB^TB^T&=&  (\alpha+1)&w^T&A^TB^T  &+&  (-\alpha)&w^T&A^T  &+& 0\cdot&w^T,\\[4mm]
  v&BAA&=&  (\beta+1)&v&BA  &+&  (-\beta)&v&B  &+&  0\cdot&v,\\[1mm]
  w^T&B^TA^TA^T&=&  (\beta+1)&w^T&B^TA^T  &+&  (-\beta)&w^T&B^T  &+& 0\cdot&w^T,\\[4mm]
  v&BAB&=&  (\alpha+1)&v&AB  &+&  0\cdot&v&B  &+&  -\alpha&v,\\[1mm]
  w^T&B^TA^TB^T&=&  (\alpha+1)&w^T&A^TB^T  &+&  0\cdot&w^T&B^T  &+&  (-\alpha)&w^T,\\[4mm]
  v&BBA&=&  (\alpha+1)&v&BA  &+&  (-\alpha)&v&A  &+&  0\cdot&v,\\[1mm]
  w^T&B^TB^TA^T&=&  (\alpha+1)&w^T&B^TA^T  &+&  (-\alpha)&w^T&A^T  &+&  0\cdot&w^T,\\[4mm]
  v&BBB&=&  0\cdot&v&BB  &+&  (\alpha^2+\alpha+1)&v&B  &+&  (-\alpha^2-\alpha)&v,\\[1mm]
  w^T&B^TB^TB^T&=&  0\cdot&w^T&B^TB^T  &+&  (\alpha^2+\alpha+1)&w^T&B^T  &+&  (-\alpha^2-\alpha)&w^T.
\end{array}
\end{equation*}
\end{lemma}
The proof is too trivial and tiresome to reproduce here.\qed

The proof of Theorem~\ref{thm_reflection} is by induction.
Set
$A(\tO)=\left(\begin{smallmatrix}1&0&0\\\alpha&\beta&\gamma\\0&0&1\end{smallmatrix}\right)$ and
$A(\tL)=\left(\begin{smallmatrix}\alpha&\beta&\gamma\\0&1&0\\0&0&1\end{smallmatrix}\right)$.
  As in~\cite{S2017}, we have for odd $n\geq 3$ such that $n=(\varepsilon_\nu\cdots \varepsilon_0)_2$,
\begin{equation}\label{eqn_transition_matrices}
x_n
=
  \left(\begin{matrix}\alpha&\beta&\gamma\end{matrix}\right)
A(\varepsilon_1)\cdots A(\varepsilon_{\nu-1})
  \left(\begin{matrix}1&1&1\end{matrix}\right)^T
\end{equation}
and the statement of the theorem is equivalent to the assertion that
\begin{equation}\label{eqn_reflection_matrix_product}
  \left(\begin{matrix}\alpha&\beta&\gamma\end{matrix}\right)
A(\varepsilon_1)\cdots A(\varepsilon_{\nu-1})
  \left(\begin{matrix}1&1&1\end{matrix}\right)^T
=
  \left(\begin{matrix}1&1&1\end{matrix}\right)
A(\varepsilon_1)^T\cdots A(\varepsilon_{\nu-1})^T
  \left(\begin{matrix}\alpha&\beta&\gamma\end{matrix}\right)^T
\end{equation}
for all $\nu\geq 1$ and all finite sequences $(\varepsilon_1,\ldots,\varepsilon_{\nu-1})$ in $\{\tO,\tL\}$.
This can be checked for $\nu\leq 3$ by simple calculation.
Let therefore $\nu\geq 4$.
Assume that $\varepsilon_1\varepsilon_2\varepsilon_3=\tO\tO\tO$.
We consider the first pair of identities in Lemma~\ref{lem_reflection_matrices}.
We multiply the first of these equations by $A(\varepsilon_4)\cdots A(\varepsilon_{\nu-1})w$ from the right and the second one by $A(\varepsilon_4)^T\cdots A(\varepsilon_{\nu-1})^Tv^T$, also from the right.
Then the left hand sides give the two constituents of~\eqref{eqn_reflection_matrix_product}, and the right hand sides are equal by the induction hypothesis.
The other $7$ cases are analogous, and the proof Theorem~\ref{thm_reflection} is complete. \qed
\bibliographystyle{siam}
\bibliography{Schnabeltier}
\end{document}